\documentclass[reqno]{amsart}

\usepackage{amsmath,amssymb,amscd,accents,enumitem}
\usepackage{graphicx}
\usepackage[utf8]{inputenc} 

\usepackage[T2A,T1]{fontenc}

\usepackage{etoolbox}
\pretocmd{\section}{\numberwithin{Cor}{section}}{}{}
\pretocmd{\subsection}{\numberwithin{Cor}{subsection}}{}{}
\pretocmd{\subsubsection}{\numberwithin{Cor}{subsubsection}}{}{}

\DeclareSymbolFont{cyrillic}{T2A}{cmr}{m}{n}
\SetSymbolFont{cyrillic}{bold}{T2A}{cmr}{bx}{n}
\DeclareMathSymbol{\B}{\mathord}{cyrillic}{193}
\def\so{\raisebox{.5ex}{\scalebox{0.6}{\#}}\kern-.02em{}o}

\DeclareGraphicsExtensions{.eps} \usepackage{mathrsfs}
\usepackage[mathcal]{eucal} 
\usepackage{esint}
\usepackage[numbers,sort]{natbib}

\usepackage[pdfencoding=unicode, psdextra, pdfusetitle]{hyperref}
\hypersetup{
  colorlinks   = true,
  linkcolor    = [RGB]{240,50,240},
  citecolor    = [RGB]{100,0,200}
}

\usepackage{cleveref}

\usepackage{bookmark}
\renewcommand{\paragraph}[1]{\medskip\noindent#1}
\renewcommand{\subparagraph}[1]{\medskip\noindent\textit{#1}}
\newcommand{\mycite}[2]{\cite[#1]{#2}}

\newtheorem{Thm}{Theorem}{\bfseries}{\itshape}
\newtheorem{CorThm}[Thm]{Corollary}{\bfseries}{\itshape}
\newtheorem{Cor}{Corollary}{\bfseries}{\itshape}
\newtheorem{ThmCor}[Cor]{Theorem}{\bfseries}{\itshape}
\newtheorem{Prop}[Cor]{Proposition}{\bfseries}{\itshape}
{\bfseries}{\itshape}
{\bfseries}{\itshape}
{\bfseries}{\itshape}
\newtheorem{Conj}[Cor]{Conjecture}{\bfseries}{\itshape}
\newtheorem{Def}[Cor]{Definition}{\bfseries}{\rmfamily}
\newtheorem{Notation}[Cor]{Notation}{\bfseries}{\rmfamily}
{\scshape}{\rmfamily}
\newtheorem{Rem}[Cor]{Remark}{\scshape}{\rmfamily}
\newtheorem*{Rem*}{Remark}{\scshape}{\rmfamily}
{\scshape}{\rmfamily}
\newtheorem{Claim}[Cor]{Claim}{\bfseries}{\itshape}

\newcommand{\iref}[2][]{%
  \def\tempA##1 ##2\relax{%
    \ifstrequal{##1}{Equation}%
        {\eqref{#1#2}}{%
    \ifstrequal{##1}{Section}
        {\S\ref{#1#2}}{%
    \ifstrequal{##1}{Subsection}
        {\S\ref{#1#2}}{%
    \ifstrequal{##1}{Subsubection}
        {\S\ref{#1#2}}%
        {##1~\ref{#1#2}}%
  }}}}%
  \expandafter\tempA\string#2\relax
}

\begin{document}
\title[Polynomial Wolff Axioms For \texorpdfstring{$\delta$}{delta}-Tubes With \#o-minimality]{Establishing the Polynomial Wolff Axioms for \texorpdfstring{$\delta$}{delta}-Separated \texorpdfstring{$\delta$}{delta}-Tubes With \#o-minimality}
\author{Gal Binyamini, Yuval Salant} 
\address{Weizmann Institute of Science, Rehovot, Israel}
\email{gal.binyamini@weizmann.ac.il}
 \thanks{
  Funded by the European Union (ERC, SharpOS, 101087910), and by the
  ISRAEL SCIENCE FOUNDATION (grant No. 2067/23).}
\date{\today}

\begin{abstract}
We establish the full version of a conjecture of Guth and Zahl, giving a lower bound for the volume of a semialgebraic set that has a large intersection with a collection of $\delta$-separated $\delta$-tubes. Our proof uses o-minimal methods to simplify the proof of Katz and Rogers, who proved the conjecture up to a small factor. We also establish that the constants depend polynomially on the complexity of the semialgebraic set, and more generally in the \so-minimal setting.
\end{abstract}

\maketitle

\section{Introduction}

\subsection{Main Results}

In \cite{KatzRogers2018OnThePolynomialWolffAxioms}, Katz and Rogers confirm a near-optimal version of the following conjecture made by Guth and Zahl (see \cite[Conjecture 1.2]{GuthZahl2018PolynomialWolffAxiomsAndKakeyaTypeEstimatesInR4}). We say that a collection of $\delta$-tubes ($\delta$-neighborhoods of unit segments) in $\mathbb R^n$ is $\delta$-separated if the tubes point in $\delta$-separated directions. That is, the angle between the central lines of any two distinct tubes is at least $\delta$. We also say that a semialgebraic set in $\mathbb R^n$ has complexity (at most) $E$ if it can be defined using intersections and unions of polynomial inequalities in $n$ variables whose sum of degrees is at most $E$.

\begin{Conj}[$\delta$-Separation Implies the Polynomial Wolff Axioms]
Let $1\geq\lambda\geq\delta>0$, let $\mathbb T$ be a $\delta$-separated collection of $\delta$-tubes in $\mathbb R^n$, and let $S\subseteq\mathbb R^n$ be a semialgebraic set of complexity $E$. Then
\[
    |S|_{\mathbb R^n}
        \geq
    \#\left\{T\in\mathbb T:|T\cap S|_{\mathbb R^n}\geq\lambda|T|_{\mathbb R^n}\right\}\cdot\delta^{n-1}\cdot\lambda^n\cdot\frac1{\mathrm{const}_{n,E}},
\]
where $|X|_{\mathbb R^n}$ is the Lebesgue measure of $X$, $\#A$ is the number of elements in $A$, and $\mathrm{const}_{n,E}$ is a non-negative constant that depends only on $n,E$.
\end{Conj}

Katz and Rogers proved this conjecture up to a factor of $\delta^\varepsilon$ (that is, they showed that $|S|_{\mathbb R^n}\geq\#\left\{T\in\mathbb T:|T\cap S|_{\mathbb R^n}\geq\lambda|T|_{\mathbb R^n}\right\}\cdot\delta^{n-1+\varepsilon}\cdot\lambda^n\cdot\frac1{\mathrm{const}_{n,E,\varepsilon}}$). A generalized version of this conjecture was then used to establish partial results regarding the Kakeya conjecture (see, for example, \cite{HickmanRogersZhang2022ImprovedBoundsForTheKakeyaMaximalConjectureInHigherDimensions}).

In this paper, we follow the proof from \cite{KatzRogers2018OnThePolynomialWolffAxioms} and a newer version of this proof, as found in the proof of \cite[Theorem 1.4]{HickmanRogersZhang2022ImprovedBoundsForTheKakeyaMaximalConjectureInHigherDimensions}. We use the more general language of o-minimality and \#o-minimality (see \iref{Subsection motivation and definition of o-minimality and sharp o-minimality}), and by simplifying the arguments we prove the full version of the conjecture with a constant of $\frac1{\operatorname{poly}_n(E)}$, where $\operatorname{poly}_n(E)$ is a polynomial in $E$ when $n$ is fixed; see \iref{Subsection new notions}.

\begin{Rem}
For readers interested primarily in the original conjecture, it is safe to restrict attention everywhere in this paper to semialgebraic sets, as these indeed constitute an o-minimal and a \#o-minimal structure; see \cite[Example 1.5.1]{V2BinyaminiNovikovZak2026SharplyOMminimalStructuresAndSharpCellularDecomposition}. Briefly, in this case, the format of the set can be taken as the ambient dimension and the degree can be taken as the sum of the degrees in a presentation of the semialgebraic set using equalities and inequalities.
\end{Rem}

We establish the following theorem:

\begin{Thm}[Main theorem]\label{Theorem main result}
Let $1\geq\lambda,\delta>0$ and let $S\in\mathbb R^n$ be definable in an o-minimal structure. Let $\mathbb T$ be a $\delta$-separated collection of $\delta$-tubes in $\mathbb R^n$ such that
\[
    \forall T\in\mathbb T:\quad
    |T\cap S|_{\mathbb R^n}
        \geq
    \lambda|T|_{\mathbb R^n}.
\]
Then
\[
    |S|_{\mathbb R^n}
        \geq
    \frac1{C(S)}\#\mathbb T\cdot\lambda^n\cdot\delta^{n-1}.
\]
Here, $C(S)>0$ is a constant depending on $S$. If $S$ is definable in a \so-minimal structure with format $F$ and degree $E$, then the constant $C(S)$ can be taken to be $\operatorname{poly}_F(E)$.
\end{Thm}

In \cite{KatzRogers2018OnThePolynomialWolffAxioms}, Katz and Rogers used the Wongkew Lemma \cite[Main Theorem]{Wongkew1993VolumesOfTubularNeighbourhoodsOfRealAlgebraicVarieties}, in the case where $S$ is the intersection of a $\delta$-neighborhood of a degree (at most) $D$ algebraic variety of dimension $k$ with a ball of constant radius, to find an upper bound of $|S|\leq\operatorname{poly}_n(D)\cdot\delta^{n-k}$. This allowed them to prove a second conjecture by Guth, which asked for an upper bound of $\mathrm{const}_{n,D,\varepsilon}\cdot\delta^{1-k-\varepsilon}$ for the number of $\delta$-separated $\delta$-tubes that can be contained in such a set $S$. The result in our paper improves this bound by eliminating the $\varepsilon$ and establishing a constant that is polynomial in $D$ when $n$ is fixed. That is, by combining \iref{Theorem main result} with the same method as in \cite{KatzRogers2018OnThePolynomialWolffAxioms}, we get the following result:

\begin{CorThm}
Let $1\geq\delta>0$ and let $S\subseteq\mathbb R^n$ be the intersection of the $\delta$-neighborhood of an algebraic variety of dimension $k$ and degree at most $D$ with a closed ball of constant radius. Let $\mathbb T$ be a collection of $\delta$-tubes that are contained in $S$ and point in $\delta$-separated directions. Then
\[
    \#\mathbb T
        \leq
    \operatorname{poly}_n(D)\cdot\delta^{1-k}.
\]
\end{CorThm}

\subsection{Overview}\label{Subsection new notions}

We briefly address the two key points where our arguments differ from \cite{KatzRogers2018OnThePolynomialWolffAxioms}, allowing us to avoid the additional $\delta^{\varepsilon}$ term.

In the first half of \cite{KatzRogers2018OnThePolynomialWolffAxioms}, the authors integrate over a semialgebraic set. This is done in two steps: first, the Yomdin-Gromov algebraic lemma \cite[Lemma 2.3]{KatzRogers2018OnThePolynomialWolffAxioms} is used to obtain a $C^r$-parametrization. This map is then approximated by a polynomial, allowing the authors to use the Bézout theorem to obtain an effective estimate. In order for this approximation to be sufficiently accurate, it is necessary to use a large $r=\theta(\frac1\varepsilon)$ which translates into a $\mathrm{const}_\varepsilon\delta^\varepsilon$ error term. We work instead directly with the relevant semialgebraic set, replacing the use of the Bézout theorem by more general o-minimal (or \so-minimal) bounds, thus avoiding the approximation step altogether.

In the second half of \cite{KatzRogers2018OnThePolynomialWolffAxioms}, there is another step involving a dyadic pigeonhole principle which introduces another $\log\delta^{-1}$ error term. In the original argument, this was absorbed into the $\delta^\varepsilon$ error term, so there was little point in attempting to improve this step. However, we show that a more refined geometric approach can be used to eliminate this additional error term, thus recovering the full original conjecture.

It turns out that all of the arguments in the paper continue to hold in the more general context of o-minimal structures (or, for polynomial bounds on constants, \so-minimal structures), so we have chosen to present the results in this generality. We present the arguments in the \so-minimal case, the o-minimal case being entirely analogous.

\subsection{AI Disclosure}

ChatGPT was used to preform basic English and typographic edits and minor formatting. No substantive mathematical work was preformed with the help of AI systems.

\section{Preliminaries}

\subsection{Asymptotic growth}

We use the notation $\mathrm{const}_{a_1,...,a_k}$ to denote any quantity that is bounded from above by a specific function that depends only on $a_1,...,a_k$ and $\operatorname{poly}_{a_1,...,a_k}(b_1,...,b_k)$ to denote any quantity which, once $a_1,...,a_k$ are fixed, becomes bounded from above by a specific polynomial in $b_1,...,b_k$. All constants and polynomials in this text can be made explicit in principle.

\subsection{O-minimality and \#o-minimality}\label{Subsection motivation and definition of o-minimality and sharp o-minimality}

We refer the reader to \cite{Vdd1998TameTopologyAndOMinimalStructures} for an introduction to the theory of o-minimal structures, and to \cite{V2BinyaminiNovikovZak2026SharplyOMminimalStructuresAndSharpCellularDecomposition} for the theory of \so-minimal structures. All of the o-minimal structures considered in this paper are taken over the real numbers. We list below the main results needed for our paper, focusing on the case of a fixed \so-minimal structure $\mathfrak S$. We briefly indicate the corresponding statements in the o-minimal case. By \cite[Remark 1.37]{V2BinyaminiNovikovZak2026SharplyOMminimalStructuresAndSharpCellularDecomposition}, we may also assume without loss of generality that $\mathfrak S$ has sharp cell decomposition.

\begin{ThmCor}[Definable Choice, \mycite{Proposition 3.1}{V2BinyaminiNovikovZak2026SharplyOMminimalStructuresAndSharpCellularDecomposition}]\label{Theorem definable choice}
Let $X\subseteq\mathbb R^a\times\mathbb R^b$ be a definable set of degree $E$ and format $F$, and let $\operatorname{\Pi}:\mathbb R^a\times\mathbb R^b\rightarrow\mathbb R^b$ be the projection onto the last $b$ coordinates.
Then there exists a definable section $f:\operatorname{\Pi}(X)\rightarrow\mathbb R^a$ of degree $\operatorname{poly}_F(E)$ and format $\mathrm{const}_F$. That is, $f$ satisfies $(f(p),p)\in X$ for each $p\in\Pi(X)$.
\end{ThmCor}

In the o-minimal case, the same result holds without the corresponding bounds on the format and degree.

For the next theorem, we will need the notion of cellular decomposition in \so-minimality: see \cite[\S1.7]{V2BinyaminiNovikovZak2026SharplyOMminimalStructuresAndSharpCellularDecomposition} for the general definition and key properties. 

\begin{ThmCor}[$C^k$ cell decomposition]\label{Theorem Ck cell decomposition}
Let $A$ be a definable set of degree $E_A$ and format $F_A$ in a \#o-minimal structure that has \#CD, and let $k\in\mathbb N$. Then $A$ can be partitioned into $\operatorname{poly}_{F_A,k}(E_A)$ cells $B_1,...,B_l$ that are of degree $\operatorname{poly}_{F_A,k}(E_A)$ and format $\mathrm{const}_{F_A,k}$, so that each $B_i$ is a $C^k$-embedding of an open cube $(0,1)^{d_i}$ via a definable map $\psi_i$ of degree $\operatorname{poly}_{F_A,k}(E_A)$ and format $\mathrm{const}_{F_A,k}$.

Let $f:A\rightarrow\mathbb R^b$ be a definable function of degree $E_f$ and format $F_f$. Then the cell decomposition of $A$ can be chosen so that the map $f\circ\psi_i:(0,1)^{d_i}\to\mathbb R^b$ is $C^k$ for each $i$. In this case, we will have $\operatorname{poly}_{F_A,F_f,k}(E_A,E_f)$ cells of degree $\operatorname{poly}_{F_A,F_f,k}(E_A,E_f)$ and format $\mathrm{const}_{F_A,F_f,k}$ and the $C^k$-embeddings will be of degree $\operatorname{poly}_{F_A,F_f,k}(E_A,E_f)$ and format $\mathrm{const}_{F_A,F_f,k}$.
\end{ThmCor}

In the o-minimal case, the same result holds without the corresponding bounds on the format and degree.

\begin{proof}
The first statement, without the assumption of $C^k$-regularity, follows from \cite[Theorem 1.35]{V2BinyaminiNovikovZak2026SharplyOMminimalStructuresAndSharpCellularDecomposition}. One can then further decompose the cells into $C^k$-smooth cells as in \cite[\S6.2]{Coste1999Ominimal}. For the second statement, it suffices to apply the first statement to the graph of $f$.
\end{proof}

\begin{Cor}\label{Corollary finite amount of connected components}
Any definable set of format $F$ and degree $E$ has at most $\operatorname{poly}_F(E)$ connected components.
\end{Cor}

\begin{proof}
This follows from \iref{Theorem Ck cell decomposition} since each cell is connected.
\end{proof}

The verbatim analogue of this statement in the o-minimal setting is that every definable set has finitely many connected components. This is true, but in practice the following stronger statement provides a more suitable analogue.

\begin{Cor}\label{Corollary o-minimal connected components}
Let $X\subseteq\mathbb R^a\times\mathbb R^b$ be a definable set in an o-minimal structure. Then the number of connected components of any fiber
\begin{equation*}
    X_\lambda := \{ y\in\mathbb R^b : (\lambda,y)\in X \}
\end{equation*}
is uniformly bounded over all $\lambda\in\mathbb R^a$.
\end{Cor}

\subsection{Notation and definitions}

We introduce some notation and definitions that we will use throughout. For $A$ a measurable subset of $\mathbb R^k$, we use $|A|_{\mathbb R^k}$ for the standard Lebesgue measure of $A$ in $\mathbb R^k$. For $x\in\mathbb R^k$, we use $\|x\|_{\mathbb R^k}$ for the Euclidean length of $x$.

\begin{Def}[$\lambda$-segment]
Let $\lambda>0$, and let $a,d\in\mathbb R^m$. We define the $\lambda$-segment starting at $(a,0)$ with direction $(d,1)$ to be the line segment
\[
    l_{a,d}(\lambda)
        :=
    \{(a+t\cdot d,t):t\in[0,\lambda]\}
        \subseteq
    \mathbb R^{m+1}.
\]
\end{Def}

\begin{Notation}[Projection Onto Directions]
Let $l_{a,d}(\lambda)$ be a $\lambda$-segment, and let $L$ be a collection of $\lambda$-segments. We define the projections
\[
    \operatorname{\Pi}_{\mathrm{dir}}(l_{a,d}(\lambda)):=d
\]
and
\[
    \operatorname{\Pi}_{\mathrm{dir}}(L):=\{\operatorname{\Pi}_{\mathrm{dir}}(l):l\in L\}.
\]
\end{Notation}

\begin{Def}
Given $x,y\in\mathbb R^m$, we define $l_{x,y}:=\{(x+t\cdot y,t):t\in\mathbb R\}\subseteq\mathbb R^{m+1}$ to be the (infinite) line passing through $(x,0)$ and pointing in the direction $(y,1)$.
\end{Def}

\begin{Def}[$\beta\times\delta$-tube]
Let $\beta>0,\delta>0$ and let $a\in\mathbb R^m,d\in[0,1]^m$. We define the $\beta\times\delta$-tube centered at $(a,0)$ with direction $(d,1)$ to be
\[
   T_{a,d}(\beta,\delta)
       :=
   \{(x,t)\in\mathbb R^m\times[0,\beta]:\|x-(a+t\cdot d)\|_{\mathbb R^m}\leq\delta\}.
\]
\end{Def}

\begin{Def}[$\delta$-tube]
Let $\delta>0$ and let $l\subseteq\mathbb R^n$ be a line segment of Euclidean length 1. We define $T_l(\delta)\subseteq\mathbb R^n$, the $\delta$-tube centered around $l$, to be the $\delta$-neighborhood of $l$.
\end{Def}

\section{Auxiliary Propositions}

We present the two propositions and one corollary used to establish \iref{Theorem main result}. In the following sections, we will prove them and show that \iref{Theorem main result} follows from them.

\begin{Prop}\label{Proposition union of line segments has large measure}
Let $L$ be a definable collection of $\lambda$-segments. Assume $L$ has degree $E$ and format $F$. Then
\[
    \left|\bigcup_{l_{a,d}(\lambda)\in L}l_{a,d}(\lambda)\right|_{\mathbb R^{m+1}}
        \geq
    \frac{1}{\operatorname{poly}_F(E)}\lambda^{m+1}|\operatorname{\Pi}_{\mathrm{dir}}(L)|_{\mathbb R^m}.
\]
\end{Prop}

For \iref{Proposition S contains segments in every direction} and \iref{Corollary tube intersection S has lambda segments}, we will use the notation $\overline{B_{x, r}^{\mathbb R^m}}$ to denote the closed ball in $\mathbb R^m$ centered at $x$ of radius $r$, and the notation $\operatorname{Tr}_{t}$ to denote the translation operator by $t$ along the $(m+1)$-th coordinate. That is,
\[
    \operatorname{Tr}_t(x_0,...,x_m,x_{m+1}):=(x_0,...,x_m,x_{m+1}+t).
\]

\begin{Prop}\label{Proposition S contains segments in every direction}
Let $1\geq\lambda,\delta>0$ and let $S\in\mathbb R^{m+1}$ be a definable set of degree $E$ and format $F$. Assume that for some $1\times\delta$-tube $T_{a_0,d_0}(1, \delta)\subseteq\mathbb R^{m+1}$ and for some $t_0\in\mathbb R$, we have that $|S\cap \operatorname{Tr}_{t_0}\left(T_{a_0,d_0}(1, \delta)\right)|_{\mathbb R^{m+1}}\geq\lambda|T_{a_0,d_0}(1, \delta)|_{\mathbb R^{m+1}}$. Then,
for any $d\in\overline{B_{d_0, \frac{\delta}{2}}^{\mathbb R^m}}$, $S$ must contain some translation along the $(m+1)$-th coordinate of a $(\frac1{\operatorname{poly}_F(E)}\lambda)$-segment pointing in the direction $d$ (the constant being independent of $d$).
\end{Prop}

Note that this proposition implies the following corollary:

\begin{Cor}\label{Corollary tube intersection S has lambda segments}
Let $h_{F,E}$ be half of the $\frac1{\operatorname{poly}_F(E)}$-constant from \iref{Proposition S contains segments in every direction}. We define for each $i\in\mathbb Z$:
\[
    S^i
        :=
    \operatorname{Tr}_{-i\cdot(h_{F,E}\lambda)}(S)\bigcap\left(\mathbb R^m\times\left[0,h_{F,E}\lambda\right]\right)
\]
(so that the translations of $S^i$ form a decomposition of $S$ along the $(m+1)$-th coordinate), and
\[
    L^i
        :=
    \{l_{a,d}\left(h_{F,E}\lambda\right):l_{a,d}\left(h_{F,E}\lambda\right)\subseteq S^i\}.
\]

Then
\[
    \bigcup_{i\in\mathbb Z}\operatorname{\Pi}_{\mathrm{dir}}(L^i)
        \supseteq
    \overline{B_{d_0, \frac{\delta}{2}}^{\mathbb R^m}}.
\]
\end{Cor}

\begin{proof}
By \iref{Proposition S contains segments in every direction}, we have, for each $\alpha\in\overline{B_{d_0, \frac{\delta}{2}}^{\mathbb R^m}}$, that $S$ must contain a translation along the $(m+1)$-th coordinate of a $(2h_{F,E}\lambda)$-segment pointing in the direction $\alpha$. Thus, $\alpha$ will belong to $\operatorname{\Pi}_{\mathrm{dir}}(L^i)$ for some $i$, and the result follows.
\end{proof}

\section{Proof of \texorpdfstring{\iref{Proposition union of line segments has large measure}}{the First Auxiliary Proposition}}

\begin{proof}
We think of $\operatorname{\Pi}_{\mathrm{dir}}$ as a definable map from $\mathbb R^m\times\mathbb R^m$ to $\mathbb R^m$. Notice that $\operatorname{\Pi}_{\mathrm{dir}}(L)$ is of degree $\operatorname{poly}_F(E)$ and format $\mathrm{const}_F$. By \iref{Theorem definable choice} (definable choice), we may choose a section
\[
    f:\operatorname{\Pi}_{\mathrm{dir}}(L)\rightarrow\mathbb R^m
\]
such that for every $d\in\operatorname{\Pi}_{\mathrm{dir}}(L)$ we have $l_{f(d),d}\in L$. Moreover, the degree of $f$ is $\operatorname{poly}_F(E)$ and the format is $\mathrm{const}_F$. We define $g:[0,\lambda]\times\operatorname{\Pi}_{\mathrm{dir}}(L)\rightarrow\mathbb R^m$ by
\begin{equation*}
    g_t(x) = g(t,x) := f(x)+t\cdot x.
\end{equation*}
Note that for each $t$, the degree of the map $g_t$ is at most $\operatorname{poly}_F(E)$ and the format is $\mathrm{const}_F$ (independently of $t$). We calculate as at the end of the proof of \cite[Theorem 3.1]{KatzRogers2018OnThePolynomialWolffAxioms}, using Fubini's theorem:
\begin{equation}\label{Equation integral dt}\begin{aligned}
    \left|\bigcup_{l_{a,d}(\lambda)\in L}l_{a,d}(\lambda)\right|_{\mathbb R^{m+1}}
        &\geq
    \left|\bigcup_{x\in\operatorname{\Pi}_{\mathrm{dir}}(L)}l_{f(x),x}(\lambda)\right|_{\mathbb R^{m+1}}
        \\&=
    \int_0^\lambda\left|g_t(\operatorname{\Pi}_{\mathrm{dir}}(L))\right|_{\mathbb R^m}dt.
\end{aligned}\end{equation}

We therefore seek, for each $t$, a lower bound for $\left|g_t(\operatorname{\Pi}_{\mathrm{dir}}(L))\right|_{\mathbb R^m}$.

By \iref{Theorem Ck cell decomposition} ($C^k$ cell decomposition), we may split $\operatorname{\Pi}_{\mathrm{dir}}(L)$ into $\operatorname{poly}_F(E)$ cells, where $f$ is $C^1$ on each cell, with each cell (uniformly) having degree $\operatorname{poly}_F(E)$ and format $\mathrm{const}_F$. By the pigeonhole principle, one of these cells, $C$, must have measure at least
\begin{equation}\label{Equation C has large measure}
    |C|_{\mathbb R^m}\geq\frac{|\operatorname{\Pi}_{\mathrm{dir}}(L)|_{\mathbb R^m}}{\operatorname{poly}_F(E)}.
\end{equation}

Let $y\in g_t\left(\operatorname{\Pi}_{\mathrm{dir}}(L)\right)$. Consider the fiber $g_t^{-1}(y)$ over $y$. It is of degree $\operatorname{poly}_F(E)$ and format $\mathrm{const}_F$. Thus, by \iref{Corollary finite amount of connected components}, it must be a union of at most $\operatorname{poly}_F(E)$ connected components (independently of $y$ or $t$).

We wish to relate $|g_t(C)|_{\mathbb R^m}$ to the integral $\int_{x\in C}|\det(D_x(g_t))|dx$, where $D_x$ is the differential at $x$. Every point $x\in C$ where $\det(D_x(g_t))\neq 0$ must be an isolated point in the fiber over $g_t(x)$ (and thus such $x$'s can cover $g_t(x)$ at most $\operatorname{poly}_F(E)$ times). Also, the points $x\in C$ where $\det(D_x(g_t))=0$ do not contribute to the integral. Thus, we have
\begin{align*}
    |g_t(\operatorname{\Pi}_{\mathrm{dir}}(L))|_{\mathbb R^m}
        &\geq
    |g_t(C)|_{\mathbb R^m}
        \\&\geq
    \frac{1}{\operatorname{poly}_F(E)}
    \int_{x\in C}
    |\det(D_x(g_t))|
    dx,
\end{align*}
and by substituting this in \iref{Equation integral dt} and using Fubini's theorem, we obtain
\begin{equation}\label{Equation bound with det}\begin{aligned}
    \left|\bigcup_{l_{a,d}(\lambda)\in L}l_{a,d}(\lambda)\right|_{\mathbb R^{m+1}}
        &\geq
    \frac{1}{\operatorname{poly}_F(E)}
    \int_0^\lambda\int_{x\in C}|\det(D_x(g_t))|dxdt
        \\&=
    \frac{1}{\operatorname{poly}_F(E)}
    \int_{x\in C}\int_0^\lambda|\det(D_x(g_t))|dtdx.
\end{aligned}\end{equation}

We calculate:
\[
    D_x(g_t)
        =
    D_x(f(x)+t\cdot x)
        =
    D_x(f)
        +
    t\cdot I_{m\times m},
\]
where $I$ is the identity matrix.

So, we see for each $x$ that $\det(D_x(g_t))$ is a monic polynomial in $t$ of degree $m$, and thus, by excluding neighborhoods of radius $\frac{\lambda}{4m}$ around its zeros, we see that
\[
    \int_0^\lambda|\det(D_x(g_t))|dt
        \geq
    \frac{\lambda}{2}\cdot\left(\frac{\lambda}{4m}\right)^m
        =
    \frac{\lambda^{m+1}}{\mathrm{const}_{m}}.
\]

Combining this with \iref{Equation bound with det} and \iref{Equation C has large measure}, we obtain
\begin{align*}
    \left|\bigcup_{l_{a,d}(\lambda)\in L}l_{a,d}(\lambda)\right|_{\mathbb R^{m+1}}
        &\geq
    \frac{1}{\operatorname{poly}_F(E)}\lambda^{m+1}
    \int_{x\in C} 1dx
        \\&=
    \frac{1}{\operatorname{poly}_F(E)}\lambda^{m+1}|C|_{\mathbb R^m}
        \\&\geq
    \frac{1}{\operatorname{poly}_F(E)}\lambda^{m+1}|\operatorname{\Pi}_{\mathrm{dir}}(L)|_{\mathbb R^m}.
\end{align*}
\end{proof}

\section{Proof of \texorpdfstring{\iref{Proposition S contains segments in every direction}}{the Second Auxiliary Proposition}}

\begin{proof}
Let $L_d:=\{l_{x,d}(1)\subseteq\mathbb R^{m+1}:l_{x,d}\cap T_{a_0,d_0}(1,\delta)\neq\emptyset\}$ be a collection of $1$-segments in the direction $d$ that covers $T_{a_0,d_0}(1,\delta)$. Notice that $\bigcup_{l_{x,d}(1)\in L_d}l_{x,d}(1)$ is a prism whose base is the convex hull of $\overline{B_{a_0,\delta}}$ and $\overline{B_{a_0+(d_0-d),\delta}}$. Thus (since $\|d_0-d\|_{\mathbb R^m}\leq\frac{\delta}2$), the volume of $\bigcup_{l_{x,d}(1)\in L_d}l_{x,d}(1)$ is at most $\mathrm{const}_m\cdot\delta^m$. Thus, since
\begin{align*}
|S\cap\operatorname{Tr}_{t_0}\left(T_{a_0,d_0}(1, \delta)\right)|_{\mathbb R^{m+1}}
    &\geq
\lambda|T_{a_0,d_0}(1, \delta)|_{\mathbb R^{m+1}}
    \\&\geq
\frac1{\mathrm{const}_m}\cdot\lambda\cdot\delta^m,
\end{align*}
and since $\bigcup_{l_{x,d}(1)\in L_d}\operatorname{Tr}_{t_0}(l_{x,d}(1))\supseteq \operatorname{Tr}_{t_0}\left(T_{a_0,d_0}(1, \delta)\right)$, we obtain
\[
    \left|S\cap\bigcup_{l_{x,d}(1)\in L_d}\operatorname{Tr}_{t_0}(l_{x,d}(1))\right|_{\mathbb R^{m+1}}
        \geq
    \frac{\lambda}{\mathrm{const}_m}\cdot
    \left|\bigcup_{l_{x,d}(1)\in L_d}\operatorname{Tr}_{t_0}(l_{x,d}(1))\right|_{\mathbb R^{m+1}}.
\]

By the pigeonhole principle, there is some $l_{x_0,d}(1)\in L_d$ such that
\[
    |\{t:(x_0+t\cdot d,t_0+t)\in S\}|_{\mathbb R}
        \geq
    \frac{\lambda}{\mathrm{const}_m}.
\]

Since the set $\{t:(x_0+t\cdot d,t_0+t)\in S\}$ is of degree $\operatorname{poly}_F(E)$ and format $\mathrm{const}_F$, by \iref{Corollary finite amount of connected components}, it can have at most $\operatorname{poly}_F(E)$ connected components. So, by the pigeonhole principle, it must contain a closed segment of length at least $\frac{\lambda}{\operatorname{poly}_F(E)}$.
\end{proof}

\section{Proof of \texorpdfstring{\iref{Theorem main result}}{the Main Theorem}}

We define $m:=n-1$ (so that $\mathbb R^n=\mathbb R^{m+1}$).

\begin{Claim}
We are allowed to make the assumption (and indeed make it) that $\mathbb T$ actually consists of translations along the $(m+1)$-th coordinate of $1\times\delta$-tubes that satisfy
\begin{align*}
    \forall\operatorname{Tr}_x(T)\in\mathbb T:&\\
    &|S\cap\operatorname{Tr}_x(T)|_{\mathbb R^{m+1}}
        \geq
    \lambda|T|_{\mathbb R^{m+1}}
\end{align*}
and
\begin{align*}
    \forall\,\operatorname{Tr}_{x_1}(T_{a_1,d_1}(1,\delta))\neq\operatorname{Tr}_{x_2}(T_{a_2,d_2}(1,\delta))\in\mathbb T:&\\
    &\|d_1-d_2\|_{\mathbb R^m}
        \geq
    2\delta.
\end{align*}
\end{Claim}

\begin{proof}
By the pigeonhole principle, we may choose some rotation of $\mathbb R^{m+1}$ and some subset of $\mathbb T$ of size at least $\frac1{\mathrm{const}_m}\#\mathbb T$ such that for each tube $T_l$ in this subset, the direction of $l$ (not normalized) may be written as an element of $[0,1]^m\times\{1\}$. For this reason, we are allowed to assume in this proof that for each tube $T_l\in\mathbb T$, the unit line segment $l$ points in a direction that belongs to the set $[0,1]^m\times\{1\}$. We now take some cover of $T_l$ with a $\mathrm{const}_m$ amount of translations along the $(m+1)$-th coordinate of $1\times\delta$-tubes that point in the direction of $l$. Since the sum of the volumes of these shapes is at most $\mathrm{const}_m\cdot|T_l|_{\mathbb R^{m+1}}$, we get that one of these shapes, denoted by $T_l'$, must satisfy that $|T_l'\cap S|_{\mathbb R^{m+1}}\geq\frac1{\mathrm{const_m}}\cdot\lambda|T_l'|_{\mathbb R^{m+1}}$. This means that by modifying $\lambda$, we can also assume in this proof that each tube is actually a translation along the $(m+1)$-th coordinate of a $1\times\delta$-tube. Finally, using a greedy strategy, we choose a subset of $\mathbb T$ of size at least $\frac1{\mathrm{const}_m}\#\mathbb T$ such that the angle between any two tubes in the subset is at least $\mathrm{const}_m\cdot\delta$, where we choose this constant appropriately so that this will imply that $\|d_1-d_2\|_{\mathbb R^m}\geq 2\delta$ (where the tubes point in the directions $(d_1,1),(d_2,1)$), and we are done.
\end{proof}

We now proceed to prove the main theorem.

\begin{proof}[Proof of \iref{Theorem main result}]
We use the definitions of $S^i$, $L^i$, and $h_{F,E}$ from the statement of \iref{Corollary tube intersection S has lambda segments}.

By \iref{Corollary tube intersection S has lambda segments}, we have that for each tube $\operatorname{Tr}_{x}(T_{a,d}(1,\delta))\in\mathbb T$:
\[
    \bigcup_{i\in\mathbb Z}\operatorname{\Pi}_{\mathrm{dir}}(L^i)
        \supseteq
    \overline{B_{d, \frac{\delta}{2}}^{\mathbb R^m}}.
\]

Since any two distinct tubes in $\mathbb T$ satisfy $\|d_1-d_2\|_{\mathbb R^m}\geq 2\delta$, we have
\[
    \sum_{i\in\mathbb Z}\left|\operatorname{\Pi}_{\mathrm{dir}}(L^i)\right|_{\mathbb R^m}
        \geq
    \left|\bigcup_{i\in\mathbb Z}\operatorname{\Pi}_{\mathrm{dir}}(L^i)\right|_{\mathbb R^m}
        \geq
    \frac1{\mathrm{const}_m}\cdot\#\mathbb T\cdot\delta^m.
\]

Notice that the degree of $L^i$ is $\operatorname{poly}_F(E)$ (uniformly in $i$) and the format is $\mathrm{const}_F$ (uniformly in $i$), and that $L^i$ is a collection of $(h_{F,E}\lambda)$-segments that are contained in $S^i$. Applying \iref{Proposition union of line segments has large measure}, we obtain
\begin{align*}
    |S^i|_{\mathbb R^{m+1}}
        &\geq
    \left|\bigcup_{l_{a,d}(h_{F,E}\lambda)\in L^i}l_{a,d}(h_{F,E}\lambda)\right|_{\mathbb R^{m+1}}
        \\&\geq
    \frac{1}{\operatorname{poly}_F(E)}\left(h_{F,E}\lambda\right)^{m+1}|\operatorname{\Pi}_{\mathrm{dir}}(L^i)|_{\mathbb R^m}.
\end{align*}

Combining these two equations, we obtain
\begin{align*}
    |S|_{\mathbb{R}^n} 
    = |S|_{\mathbb{R}^{m+1}}
    = \sum_{i \in \mathbb{Z}} |S^i|_{\mathbb{R}^{m+1}}
    &\geq \frac{1}{\operatorname{poly}_F(E)} \left(h_{F,E} \lambda \right)^{m+1} \cdot \sum_{i \in \mathbb{Z}} |\operatorname{\Pi}_{\mathrm{dir}}(L^i)|_{\mathbb{R}^m} \\
    &\geq \frac{1}{\operatorname{poly}_F(E)} \cdot \#\mathbb{T} \cdot \lambda^{m+1} \cdot \delta^m
    \\&= \frac{1}{\operatorname{poly}_F(E)} \cdot \#\mathbb{T} \cdot \lambda^n \cdot \delta^{n-1}.
\end{align*}
\end{proof}

\bibliographystyle{plainnat}
\bibliography{references}

\begin{thebibliography}{7}
\providecommand{\natexlab}[1]{#1}
\providecommand{\url}[1]{\texttt{#1}}
\expandafter\ifx\csname urlstyle\endcsname\relax
  \providecommand{\doi}[1]{doi: #1}\else
  \providecommand{\doi}{doi: \begingroup \urlstyle{rm}\Url}\fi

\bibitem[Binyamini et~al.(2026)Binyamini, Novikov, and Zak]{V2BinyaminiNovikovZak2026SharplyOMminimalStructuresAndSharpCellularDecomposition}
G.~Binyamini, D.~Novikov, and B.~Zak.
\newblock Sharply o-minimal structures and sharp cellular decomposition.
\newblock \emph{arXiv:2209.10972v2}, 2026.
\newblock \doi{10.48550/arXiv.2209.10972}.

\bibitem[Coste(1999)]{Coste1999Ominimal}
Michel Coste.
\newblock An introduction to o-minimal geometry.
\newblock HAL, 1999.
\newblock URL \url{https://hal.science/hal-05413940}.
\newblock HAL: hal-05413940.

\bibitem[Guth and Zahl(2018)]{GuthZahl2018PolynomialWolffAxiomsAndKakeyaTypeEstimatesInR4}
L.~Guth and J.~Zahl.
\newblock Polynomial wolff axioms and kakeya-type estimates in {$\mathbb R^4$}.
\newblock \emph{Proceedings of the London Mathematical Society}, 117:\penalty0 192--220, 2018.
\newblock \doi{10.1112/plms.12138}.

\bibitem[Hickman et~al.(2022)Hickman, Rogers, and Zhang]{HickmanRogersZhang2022ImprovedBoundsForTheKakeyaMaximalConjectureInHigherDimensions}
J.~Hickman, K.M. Rogers, and R.~Zhang.
\newblock Improved bounds for the kakeya maximal conjecture in higher dimensions.
\newblock \emph{American Journal of Mathematics}, 144\penalty0 (6):\penalty0 1511--1560, 2022.
\newblock \doi{10.1353/ajm.2022.0037}.

\bibitem[Katz and Rogers(2018)]{KatzRogers2018OnThePolynomialWolffAxioms}
N.H. Katz and K.M. Rogers.
\newblock On the polynomial wolff axioms.
\newblock \emph{Geometric and Functional Analysis}, 28:\penalty0 1706–1716, 2018.
\newblock \doi{10.1007/s00039-018-0466-7}.

\bibitem[van~den Dries(1998)]{Vdd1998TameTopologyAndOMinimalStructures}
L.P.D. van~den Dries.
\newblock \emph{Tame Topology and O-minimal Structures}.
\newblock Cambridge University Press, 1998.
\newblock \doi{10.1017/CBO9780511525919}.

\bibitem[Wongkew(1993)]{Wongkew1993VolumesOfTubularNeighbourhoodsOfRealAlgebraicVarieties}
R.A. Wongkew.
\newblock Volumes of tubular neighbourhoods of real algebraic varieties.
\newblock \emph{Pacific Journal of Mathematics}, 159\penalty0 (1):\penalty0 177–184, 1993.
\newblock \doi{10.2140/pjm.1993.159.177}.

\end{thebibliography}

\end{document}